\documentclass{article}
\usepackage{graphicx} 

\usepackage[margin=1in]{geometry} 

\usepackage{amsmath}
\usepackage{amssymb}
\usepackage{amsthm}
\usepackage{amscd}
\usepackage{mathtools}
\usepackage{graphicx}
\usepackage{enumitem}
\usepackage{faktor}
\usepackage{MnSymbol} 
\usepackage{verbatim}
\usepackage{marginnote}
\usepackage{todonotes}
\usepackage{bbm}
\usepackage{bm}
\usepackage{blindtext}
\usepackage{multicol}
\usepackage{color}
\usepackage{xcolor}
\usepackage{authblk}
\usepackage[colorlinks=true,linkcolor=black,citecolor=black,urlcolor=black]{hyperref} 

\usepackage{tikz-cd} 
\tikzcdset{scale cd/.style={every label/.append style={scale=#1}, cells={nodes={scale=#1}}}}
\usetikzlibrary{arrows}

\usepackage[style=alphabetic]{biblatex} 
\renewbibmacro{in:}{} 

\newtheorem{theorem}{Theorem}[section]
\newtheorem*{theorem*}{Theorem}

\newtheorem*{conjecture*}{Conjecture}
\newtheorem{question}[theorem]{Question}
\newtheorem*{question*}{Question}

\newtheorem*{guess*}{Guess}
\newtheorem*{Assumption*}{Assumption}

\newtheorem*{problem*}{Problem}
\newtheorem{lemma}[theorem]{Lemma}
\newtheorem*{lemma*}{Lemma}
\newtheorem{proposition}[theorem]{Proposition}
\newtheorem*{proposition*}{Proposition}
\newtheorem{corollary}[theorem]{Corollary}
\newtheorem*{corollary*}{Corollary}

\theoremstyle{definition}
\newtheorem{definition}[theorem]{Definition}
\newtheorem*{definition*}{Definition}

\newtheorem{remark}[theorem]{Remark}
\newtheorem*{remark*}{Remark}

\newtheorem{example}[theorem]{Example}

\newtheorem*{example*}{Example}
\newtheorem*{examples*}{Examples}

\newcommand{\CC}{\mathbb{C}}

\newcommand{\QQ}{\mathbb{Q}}

\newcommand{\ZZ}{\mathbb{Z}}

\usepackage{multirow}
\usepackage{multicol}
\usepackage{pythonhighlight}

\title{Global Representation Ring and Knutson Index}
\author[1]{Dylan Johnston\thanks{Email: \texttt{dylan.johnston@warwick.ac.uk}, ORCiD: 0009-0004-4893-9305}}
\author[1]{Diego Martín Duro\thanks{Email: \texttt{diego.martin-duro@warwick.ac.uk}, ORCiD: 0000-0001-9349-2427}}
\author[1]{Dmitriy Rumynin\thanks{Email: \texttt{d.rumynin@warwick.ac.uk}, ORCiD: 0000-0001-9507-3058}}

\affil[1]{Department of Mathematics, University of Warwick, Coventry, CV4 7AL, UK}

\date{\today}

\begin{document}

\maketitle
\setlength\parindent{0pt}

\tableofcontents

\begin{abstract}
Global representation rings were discovered by Sarah Witherspoon in 1995 and the Knutson Index was introduced by the second author in 2022. In the present paper we introduce the Knutson Index for general commutative rings and study it for Burnside rings and global representation rings. We also introduce the global table of a finite group, that encompasses both the character table and the Burnside table of marks. We discuss what properties of a group can be recovered from its global table.

\end{abstract}

\section{Introduction}
In $1973$ Donald Knutson conjectured that for every irreducible character $\chi$ of a finite group $G$ there exists a virtual character $\lambda \in \mathbb{Z}[\text{Irr}(G)]$ such that $\chi \otimes \lambda = \rho_\text{reg}$, where $\rho_{\text{reg}}$ denotes the regular character of $G$ \cite{K}. If such a $\lambda$ exists we call it a regular inverse to $\chi$. Note that regular inverses are not unique. In 1992 this conjecture was disproven by Savitskii, for instance, it does not hold for ${\mathrm SL}_2(\mathbb{F}_5)$ \cite{S}. In $2022$ this subject was revisited by the second author who defined the Knutson Index as a measure of how much the group failed to satisfy Knutson's Conjecture \cite{Diego}. Specifically, given a character $\chi$ of a group $G$ the Knutson Index of $\chi$ is defined as

$$ \mathcal{K}(\chi) = \min \big\{ n \; : \;  n>0 \ \mbox{ and } \ \chi \otimes \lambda = n\rho_{\text{reg}} \ \text{ for some } \ \lambda \in \mathbb{Z} [\text{Irr}(G)] \big\}. $$

The Knutson Index of the group $G$ is given by

$$ \mathcal{K}(G) = \text{lcm} \big\{ \mathcal{K}(\chi) \,:{\chi \in \text{Irr}(G)} \big\}. $$

In particular, a group $G$ satisfies Knutson's Conjecture if and only if $\mathcal{K}(G) = 1$.\\ 

The notion of the Knutson Index is generalised by the second author \cite{Diego2}. The definition is extended to allow inverting with respect to characters other than the regular one. The current paper builds on this idea. In Section~$2$ we define the Knutson Index for finitely generated commutative rings and define a subindex, that is used in later sections. The main result is essentially an algorithm for computing the Knutson Index of a reduced ring, that is implicit in Proposition~\ref{nice rings have finite inverse set}.\\

For the remainder of the paper, we return to thinking about a finite group $G$. In Section~$3$ we discuss the Knutson Index for the Burnside ring of $G$. One interesting application is that the subindex $\mathcal{K}_r^S$ controls when all exact sequences of Sylow $p-$subgroups of an exact sequence of groups split (cf. Theorem~\ref{main_th}):

\begin{theorem*}
Let $1 \to N \to G \to \Gamma \to 1$ be an exact sequence. Then $\mathcal{K}_r^S(G/N) = 1$ if and only if the sequence of Sylow $p$-subgroups $1 \to N_p \to G_p \to \Gamma_p \to 1$ splits for all primes $p$, where $G_p$ is a Sylow $p$-subgroup of $G$, $N_p = N \cap G_p$ and $\Gamma_p = (G_pN)/N$. 
\end{theorem*}

In the fourth and final section, we discuss the reduced global representation ring of a finite group $G$. This is a variation of the ring introduced by Sarah Witherspoon \cite{W}. Given a normal subgroup $N$, this is the reduced Grothendieck ring of $G-$equivariant vector bundles on finite $G/N$-sets. Just as for the Burnside ring, the reduced global representation ring can be encoded in a table, which we denote by $T(G,N)$. In Propositions~\ref{description_of_blocks} and \ref{properties global table} we describe what sort of information about $G$ the table $T(G,N)$ contains. \\

We call $T(G,e)$ the global table of $G$. It contains a plethora of information about the group. For instance, it ``knows" the character table and the table of marks. It does not know everything: we have found two pairs of non-isomorphic groups of order $256$ (smallest possible) with the same global tables. As part of this search, we have found many pairs of non-isomorphic groups with the same character table and table of marks, cf. Appendix~\ref{appendix-samecharburn} (the smallest possible order is $243$ here). 
It is an interesting horizon for future research to determine the full extent of the global table's knowledge of the underlying group. \\

The global representation ring has an obvious regular element. This leads to a discussion of the Knutson Index and numerous future directions of research, which we facilitate by stating two questions throughout the paper. 

\section*{Acknowledgements}

We are grateful to Benjamin Sambale, who told us about Shemetkov's Theorem and answered Question~3.7 (now {Proposition}~\ref{Shemetkov}) from an early version of the paper. The first author was supported by the Heilbronn Institute for Mathematical Research (HIMR) and the UK Engineering and Physical Sciences Research Council (EPSRC) under the $\lq\lq$Additional Funding Programme for Mathematical Sciences" grant EP/V521917/1. The second author was supported by the UK Engineering and Physical Sciences Research Council (EPSRC) grant EP/T51794X/1. The third author is grateful to the Max Planck
Institute for Mathematics in Bonn for its hospitality and financial support.

\section{Knutson Index for Rings}

Let $R$ be a finitely generated commutative ring. 
Suppose that $R$ is equipped with a ring homomorphism $\alpha: R \to \mathbb{Z}$.  We call $\alpha$ {\em the dimension homomorphism} and $\alpha (r)$ {\em the dimension} of the element $r\in R$. We call an element $r \in R$ {\em regular} if $\alpha(r)\neq 0$ and $x\cdot r = \alpha(x) r$ for all $x \in R$. In particular, observe that any regular element $r$ does not belong to the additive torsion and satisfies $r^2 = \alpha(r)r$. 

\begin{definition}
{\em The Knutson Index} $\mathcal{K}_r(x)$ of an element $x \in R$ with respect to a regular element $r \in R$ is the non-negative integer $m$ such that $Rx \cap \ZZ r = m \ZZ r.$ The Knutson Index of the ring $R$ with fixed generating set of non-zero dimensional $\ZZ$-module generators $\{x_1,\dots,x_n\}$ is defined as $\mathcal{K}_r(R,\{x_1,\dots,x_n\}) = \text{lcm}\big\{\mathcal{K}_r(x_i)\big\}$ and the Knutson Index $\mathcal{K}_r(R)$ of the ring is the minimum over all such generating sets.
\end{definition}

Observe that $\mathcal{K}_r(x)$ always divides $\alpha(x)$. In particular, we exclude zero dimensional generators in the last definition because $\mathcal{K}_r(x)=0$, turning the least common multiples to zero.\\

One of our main examples is the character ring $R$ of a finite group $G$.
The terminology is influenced by this example:
$\alpha (\chi) = \chi (1_G)$ and $r$ is the regular character.

\begin{definition}
Let $S \subset R$ be a subset, and for any $x \in R$ define $S(x) = \{s \in S: s \cdot x \in \ZZ r \}$. Associated to $S(x)$ we also set $\text{Ind}(x) = \{m \in \ZZ: s \cdot x = m \cdot r \text{ for } s \in S(x)\}$. Define the Knutson $S-$subindex of $x \in R$ with respect to the regular element $r \in R$ as  $\mathcal{K}_r^S(x) := \text{gcd}\big(\text{Ind}(x)\big)$. By convention, if $S(x) = \emptyset$ then we set $\mathcal{K}_r^S(x) = \infty$. 
\end{definition}

Note that $\mathcal{K}^S_r(x)$ equals the smallest positive integer $m$ such that $x \cdot \sum c_i s_i = m r$ for some $c_i \in \ZZ$ and $s_i \in S(x)$.

\begin{lemma}
Let $S \subset R$ be a subset and $r \in R$ be a regular element. Then for any $x \in R$ we have that $\mathcal{K}_r(x)$ divides $\mathcal{K}_r^S(x)$.
\end{lemma}

\begin{proof}
Suppose that $\mathcal{K}_r^S(x) = n$ and $\mathcal{K}_r(x) = m$. Then $x \cdot \sum c_i y_i = n r$ for some $y_i \in S(x)$ and $c_i \in \ZZ$. Therefore $n r \in Rx$ and so $n r \in m \ZZ r$ and we conclude that $m$ divides $n$.
\end{proof}

\begin{corollary}
Let $R$ be a ring, finitely generated as a $\ZZ$-module, with regular element $r \in R$,
and dimension homomorphism $\alpha: R \rightarrow \ZZ$. Let $x \in R$ with $\alpha(x) \neq 0$. Then $\mathcal{K}_r(x)$ divides $\alpha(x)$, in particular, the Knutson Index of $x$ is finite. Moreover, the Knutson Index of $R$ is finite.
\end{corollary}

\begin{proof}
Take $S = \{ r \}$. Then $\mathcal{K}_r^S(x) = \alpha(x)$ and so by the lemma $\mathcal{K}_r(x)$ divides $\mathcal{K}_r^S(x) = \alpha(x)$. Note that $\mathcal{K}_r(R)$ divides $\text{lcm}\{\mathcal{K}_r(x_i)\}$ for any generating set $\{x_i\}$ and so it is finite.
\end{proof}

\begin{question}
Does there exist a finite subset $S \subset R$ such that for all $x \in R$ we have $\mathcal{K}_r(x) = \mathcal{K}_r^S(x)$?
\end{question}

In the special case of the next lemma we can answer the question in the affirmative.
It covers all the rings in the present paper because they all can be realised as subrings of 
$\CC^m = \CC \times \CC \times \ldots \times \CC$.

\begin{proposition}\label{nice rings have finite inverse set}
Suppose that the ring $R$ is reduced. Then there exists a finite subset $S \subset R$ such that $\mathcal{K}_r(x) = \mathcal{K}_r^S(x)$ for all $x \in R$.
\end{proposition}

\begin{proof}
Since $R$ is finitely generated, its spectrum has finitely many generic points $I_1,\ldots, I_m$. Suppose that $\ker(\alpha)$ belongs to the first $k$ components, given by $I_1,\ldots , I_k$. Each quotient $R_i=R/I_i$ is a domain. We identify $R$ with its image under the inclusion $R \to R_1 \times \ldots \times R_m$. \\

Let $r = (r_1, \ldots, r_m)$ be the regular element. For each $i\leq k$, the map $\alpha$ factors only through $R/I_i$, so $0\neq \alpha(r)=\bar{\alpha}(r_i)$ and $r_i\neq 0$.\\

In fact, $r=(r_1,0,\ldots, 0)$ and $k=1$. To prove this, consider $x = (x_1, \ldots x_m)\in R$ of dimension $d=\alpha (x)$. Then $r x = d r$ so $r_i x_i = d r_i$ for all $i$. Since $R_i$ is a domain, it follows that if $r_i \neq 0$, then $x_i=d\in R_i$. For each $i>1$, we can find a non-zero element $0\neq y\in \cap_{j\neq i}I_j$. Note that $y=(0, \ldots, 0, y_i,0,\ldots, 0)$ and $y\in I_1 \subseteq \ker (\alpha)$. It follows that $r_i=0$ for $i>1$. \\


Notice that the regularity of $r$ means that $Rr = \ZZ r$ and $R_1 r_1 = \ZZ r_1$. Thus, $\ZZ r_1$ is a simple $R_1$-module for the domain $R_1$. It follows that $R_1=\ZZ$ and $\alpha (x) = x_1$ for all $x\in R$. \\


Now, for a subset $T \subset \{2,\ldots,m\}$ let $Z(T) = \{x \in R : x_i = 0 \text{ for all } i \in T\}$. For any $x \in R$, define $T(x)$ to be the subset of $\{2,\ldots,m\}$ with $x_i \neq 0$ if and only if $i \in T(x)$. Since $Z(T(x))=\cap_{j\in T(x)}I_j$ is not a subset of $I_1$, there exists an element (not necessarily unique) $y \in Z(T(x))$ with minimal positive dimension. By definition, we have $x \cdot y = \mathcal{K}_r(x)r$. Furthermore, observe that all elements $x' \in R$ with $T(x') = (x)$ will also satisfy $x' \cdot y = \mathcal{K}_r(x')r$. Thus, define $S \subset R$ to be the finite subset of elements consisting of a minimal positive dimensional element in $Z(T)$ for each subset $T \subset \{2,\ldots,m\}$. Then $S$ satisfies the conditions of the proposition.
\end{proof}

\section{Knutson Index for Burnside Rings}\label{knutson for burnside}

Let $G$ be a finite group. The Burnside ring $\Omega(G)$ is generated by isomorphism classes of $G-$sets, with addition given by the disjoint union and multiplication given by the Cartesian product. Recall that any finite $G$-set $X$ may be expressed as a disjoint union of $G-$orbits, say $\displaystyle X = \bigcupdot_i X_i$. By the Orbit-Stabiliser Theorem, we can write $X_i = G/G_i$ where $G_i$ is the stabiliser of some $x_i \in X_i$. Note that choosing a different $x_i' \in X_i$ gives a conjugate of $G_i$. Therefore generators of $\Omega(G)$ are given by sets $G/H$ where $H$ runs over all subgroups of $G$ up to conjugation. In other words, if $S(G)$ denotes a full list of representatives of subgroups of $G$ up to conjugation then 
\[ \Omega(G) = \bigg\{ \sum_{H_i \in S(G)} c_i \cdot  G/H_i\, :\, c_i \in \ZZ \bigg \}.\]

Given a subgroup $H$ and a $G$-set $X$ we define the mark of $H$ on $X$ by $[X,H] = |X^H|$. Since the mark respects both addition and multiplication of $G-$sets and can record all information about $\Omega(G)$ in the table of marks, $T(G)$. There exists a partial ordering on $S(G)$, the set of representatives of subgroups of $G$ up to conjugation, as follows: for each $H,K \in S(G)$ we have $H \preceq K$ whenever $H$ is a subgroup of some conjugate of $K$. Now we choose a total order $S(G)=\{H_1,H_2, \ldots \}$ such that $|H_i|\leq |H_j|$ for $i\leq j$. This total order obviously extends the partial order $\preceq$. The rows of the table are labelled by $G-$sets of the form $G/H_i$ for $H_i \in S(G)$, columns are labelled by subgroups $H_j \in S(G)$, with both labellings following the total order on $S(G)$. The entries are $T(G)_{i,j} = [G/H_i, H_j]$. That is, the entry in position $(i,j)$ is the mark of $H_j$ on $G/H_i$.

\begin{lemma}
The table of marks satisfies the following properties:
\begin{enumerate}[label=\roman*)]
    \item The diagonal entries are $[G/H_i, H_i] = |N_G(H_i)/H_i|$.
    \item The table is lower triangular.
    \item Both the rows and columns form bases of $\QQ^n$, where $n = |S(G)|$. 
\end{enumerate}
\end{lemma}

\begin{proof}
All properties follow from the observation that $(G/H_i)^{H_j} = \{gH_i : H_jgH_i \subset gH_i\}$. In particular, note that if $H_j$ is not conjugate to a subgroup of $H_i$ then $T(G)_{i,j} = 0$.
\end{proof}

We now consider Knutson Indices for the Burnside ring. Throughout, we will take the rows of the table as a generating set for $\Omega(G)$, and the dimension homomorphism to be \[\alpha: \Omega(G) \longrightarrow \ZZ;\,\, \alpha(G/H) = [G/H,\{e\}] = |G/H|.\] Observe that for the $G$-set $G/H$ we have $\displaystyle G/H \times G = \bigcupdot_{i=1}^{|G/H|} G = |G/H| \cdot G$. This makes $G = G/\{e\}$ a valid choice for a regular element of $\Omega(G)$, in fact, we assume for the rest of the section that $r = G/\{e\}.$

\begin{proposition} \label{prop32}
Let $S = \{G/H_1, \ldots G/H_n\}$ be the set consisting of rows of the table. 
Then we have $G/H_i \times G/H_j = m \cdot G/\{e\}$ for some $m$ if and only if 
$H_i \cap gH_jg^{-1} = \{e\}$ for all $g\in G$.
Furthermore, 
\[
m = \frac{|G|}{|H_i||H_j|}
\quad \mbox{and} \quad  
\mathcal{K}^S_r(G/H_i) = \text{gcd} \left\{ \frac{|G|}{|H_i||H_j|} \ | \ \forall g\in G \ \ H_i \cap gH_jg^{-1} = \{e\} \right\} .
\]
\end{proposition}

\begin{proof}
For the first statement, observe that 
$H_i \cap gH_jg^{-1} = \{e\}$ for all $g\in G$
if and only if
$hH_ih^{-1} \cap gH_jg^{-1} = \{e\}$ for all $g,h\in G$
if and only if
the stabiliser of each $(hH_i, gH_j) \in G/H_i \times G/H_j$ is trivial. \\

For the second statement, taking sizes yields $|G|^2/(|H_i||H_j|) = m |G|$ and $m = {|G|}/{(|H_i||H_j|)}$.
%
The last claim follows immediately from the definition of $\mathcal{K}^S_r$.
\end{proof}

For the remainder of the section, we explore the connection between the Knutson subindex and the splitting of exact sequences of Sylow $p$-subgroups. Before embarking on this, we introduce some notation. Given an exact sequence $1 \to N \to G \to \Gamma \to 1$ and a Sylow $p-$subgroup $G_p < G$, we have Sylow $p-$subgroups $N_p \coloneqq N \cap G_p$ of $N$ and $\Gamma_p \coloneqq (G_pN)/N$ of $\Gamma$ and the short exact sequence $1 \to N_p\to G_p \to \Gamma_p \to 1$.

\begin{lemma} \label{lemma p-split}
The following statements hold for a short exact sequence $1 \to N \to G \to \Gamma \to 1$.
\begin{enumerate}[label=\roman*)]
    \item Fix a prime $p$. All exact sequences $1 \to N_p\to G_p \to \Gamma_p \to 1$ are split or non-split simultaneously.
    \item If the original sequence splits, then the sequence $1 \to N_p\to G_p \to \Gamma_p \to 1$ splits for any prime $p$ and any Sylow $p-$subgroup $G_p < G$. 
\end{enumerate}
%
\end{lemma}

\begin{proof}
Suppose one of the sequences of Sylow $p-$subgroups splits. Then there exists $K_p$, a complementary subgroup to $N_p$ in $G_p$. Any other sequence has the form $1 \to gN_pg^{-1} \to gG_pg^{-1} \to gK_pg^{-1} \to 1$. Then $gK_pg^{-1}$ is complementary to $gN_pg^{-1}$ in $gG_pg^{-1}$. \\

As the sequence splits, there exists a subgroup $K$, complementary to $N$ in $G$. Its Sylow $p-$subgroup $K_p$ sits inside a Sylow $p-$subgroup $G_p$ of $G$. The corresponding sequence $1 \to N_p\to G_p \to \Gamma_p \to 1$ splits. By part~i), all of them split.


\end{proof}

\begin{lemma} \label{lemma index}
If a short exact sequence $1 \to N \to G \to \Gamma \to 1$ splits then $\mathcal{K}^S_r(G/N) = 1$.
\end{lemma}

\begin{proof}
We use Proposition~\ref{prop32} where $H_j=N$ is normal.
The short exact sequence $1 \to N \to G \to \Gamma \to 1$ splits if and only if there exists a subgroup $H$ of $G$ such that $H \cong \Gamma$ and $N \cap H = \{e\}$ and this implies that $\mathcal{K}^S_r(G/N) = 1$.
\end{proof}

\begin{lemma} \label{nilpotent}
Let $G$ be nilpotent and $N \lhd G$ a normal subgroup. Then the following are equivalent:
\begin{enumerate}[label=\roman*)]
    \item $1 \to N \to G \to \Gamma\to 1$ splits.
    \item $1 \to N_p \to G_p \to \Gamma_p \to 1$ splits for all primes $p$.
    \item $\mathcal{K}_r^S(G/N)=1$.
\end{enumerate}
\end{lemma}

\begin{proof}
We give the proof when $G$ is a $p-$group. The result for nilpotent groups follows by writing it as a product of its Sylow $p-$subgroups. For $G$ a $p-$group $i)$ and $ii)$ are clearly equivalent. By Lemma~\ref{lemma index} it is sufficient to show that $iii)$ implies $i)$. As $G$ is a $p-$group, we have by dimension considerations that all elements of $\text{Ind}(G/N)$ are powers of $p$. It follows that $\gcd\big(\text{Ind}(G/N)\big) = \min\big(\text{Ind}(G/N)\big)$. Therefore, if $\mathcal{K}_r^S(G/N)=1$ then there exists $K < G$ such that $|K||N| = |G|$ and $K \cap N = \{e\}$ i.e. $G = N \rtimes K$ and so $1 \to N \to G \to \Gamma\to 1$ splits.
\end{proof}

\begin{remark}\label{converse_do_not_hold}
Note that the converses of Lemma~\ref{lemma p-split}({\em ii}) and Lemma~\ref{lemma index} do not hold in general. The smallest counterexample to both is the unique non-split extension $Q_8 \cdot C_6$, which is the SmallGroup$(48,33)$ in the GAP library. Further counterexamples can be found in a table of the appendix. In fact, such a group being a counterexample to the converses of both lemmas is no coincidence, as we will see in Theorem \ref{main_th}. 
\end{remark}

Note that for all counterexamples in our table, there is a prime $p$ such that $p^4$ divides the order of $G$. This is because of the following proposition that easily follows from Shemetkov's Theorem \cite{Sh}.

\begin{proposition}\label{Shemetkov}
Suppose $1 \to N \to G \to \Gamma \to 1$ is a non-split exact sequence such that $1 \to N_p \to G_p \to \Gamma_p \to 1$ splits for all primes $p$. Then there exists a prime $q$ such that $q^4$ divides $|G|$.    
\end{proposition}

We remark that this proposition was posed as a question in an earlier version of the paper (and answered by Benjamin Sambale) \cite{Sa}.

\begin{lemma}
Let $N$ be a normal subgroup of $G$. Then 
$ \displaystyle \mathcal{K}_r^S(G/N) = \prod_p \mathcal{K}_r^S(G_p/N_p)$.
\end{lemma}

\begin{proof}
We want to show that $\mathcal{K}_r^S(G_p/N_p)$ is equal to the $p-$part of $\mathcal{K}_r^S(G/N)$ for all primes $p$. Recall that for a $G$-set $G/N$ we have $S(G/N) = \big\{ K < G : N \cap K = \{e\}\big\}$ and $\text{Ind}(G/N) = \big\{m : G/N \times G/K = m \cdot G/\{e\} \text{ for } K \in S(G/N) \big\}$. Let $S(G/N) = \{K_1,\dots, K_r\}$ and $\text{Ind}(G/N) = \{m_1,\dots,m_r\}$.\\

We begin by showing that $\mathcal{K}_r^S(G_p/N_p)$ divides the $p$-part of $\mathcal{K}_r^S(G/N)$ for all primes $p$. So, fix a prime $p$ and let $d \in \ZZ_{\geq 0}$ be the largest positive integer such that $p^d$ divides $\mathcal{K}_r^S(G/N)$. Then there exists an $m_i \in \text{Ind}(G/N)$ such that $p^d$ divides $m_i$ but $p^{d+1}$ does not. Assume, without loss of generality, that $i=1$. Now, since $N \cap K_i = \{e\}$ for all $i$, we have $N_p \cap (K_i \cap G_p) = \{e\}$ and so $\{K_1 \cap G_p,\dots,K_r \cap G_p \} \subset S(G_p/N_p)$. As $\mathcal{K}_r^S(G_p/N_p)$ does not depend on which conjugate of $G_p$ (i.e. on which Sylow $p-$ subgroup) we choose we may assume that $G_p$ was chosen so that there exists a Sylow $p-$subgroup $(K_1)_p < K_1$ such that $(K_1)_p < G_p$. Then $(K_1)_p = K_1 \cap G_p$ and we have $\frac{|G_p|}{|N_p||K_1 \cap G_p|} = p^d$. Therefore we have $p^d \in \text{Ind}(G/N)$ which implies that $\mathcal{K}_r^S(G_p/N_p)$ divides the $p-$part of $\mathcal{K}_r^S(G/N)$.\\

Conversely, we show that $p^d$ divides $\mathcal{K}_r^S(G/N)$. Assume that $\mathcal{K}_r^S(G/N) = p^c$ with $c \in \ZZ_{\geq 0}$. Then there exists $L_p < G_p$ such that $L_p \cap G_p = \{e\}$ and $|N_p||L_p| = \frac{|G_p|}{p^c}$, or equivalently, $G_p/L_p \times G_p/N_p = p^c G_p/\{e\}$. 

Now consider $G/N \times G/L_p$; we have $N \cap L_p = \{e\}$ and $|N||L_p| = \frac{|G|}{p^c \cdot n}$ where $\gcd(p,n) = 1$. That is, $G/N \times G/L_p = p^c \cdot n G/\{e\}$ which implies that $p^c \cdot n \in \text{Ind}(G/N)$. Thus $p^d$ divides $p^c$ and so $\mathcal{K}_r^S(G_p/N_p)$ is equal to the $p$-part of $\mathcal{K}_r^S(G/N)$ for all primes $p$, as required.
\end{proof}

The following theorem shows that an extension of a group is a counterexample to the converse of Lemma~\ref{lemma p-split}({\em ii}) if and only if it is a counterexample to the converse of Lemma~\ref{lemma index}.

\begin{theorem} \label{main_th}
Let $1 \to N \to G \to \Gamma \to 1$ be an exact sequence. Then $\mathcal{K}_r^S(G/N) = 1$ if and only if the sequence $1 \to N_p \to G_p \to \Gamma_p \to 1$ splits for all primes $p$.
\end{theorem}

\begin{proof}
Let $p$ divide $\mathcal{K}^S_r(G/N)$, then by the previous lemma $p$ also divides $\mathcal{K}^S_r(G_p/N_p)$ and so there does not exist $K_p \leq G_p$ such that $|K_p||N_p| = |G_p|$ and $K_p \cap N_p = \{ e \}$. Conversely, let $p$ be such that $1 \to N_p \to G_p \to \Gamma_p \to 1$ does not split. Then $1 \notin$ Ind$(G_p/N_p)$ and all elements of Ind$(G_p/N_p)$ are powers of $p$, thus $p$ divides $\mathcal{K}^S_r(G_p/N_p)$ and we conclude that $p$ divides $\mathcal{K}^S_r(G/N)$.
\end{proof}

\section{The Global Representation Ring and its Knutson Index}

Let $G$ be a finite group with $H \leq G$ a subgroup. Let $V_G(G/H)$ denote the set of $G$-equivariant vector bundles on $G/H$. That is, for each $xH \in G/H$ assign a vector space $V_{xH}$ along with an action of $G$ given by $g \cdot V_{xH} = V_{gxH}$. Note that $h \cdot V_{H} = V_{H}$ and so the fibre at the identity is a $H$-module. In fact, this vector bundle is determined by its fibre at the identity coset, and so we have a one-to-one correspondence between $G$-equivariant vector bundles on $G/H$ and representations of $H$ \cite{Se}.\\

Now, given a normal subgroup $N \lhd G$, with quotient $\Gamma = G/N$ we consider all $G$-equivariant vector bundles on finite $\Gamma$-sets. Recall from Section \ref{knutson for burnside} that all finite $\Gamma$-sets are given by disjoint unions of sets of the form $\Gamma/K$ where $K$ is a subgroup of $\Gamma$. By the third isomorphism theorem, any $\Gamma$-set $\Gamma/K$ is isomorphic to $G/K'$ where $K'$ is the preimage of $K$ under the natural quotient $G \to \Gamma$. It follows that the $G$-equivalent vector bundles on finite $\Gamma$-sets are generated by $G-$equivariant vector bundles on $G/H$ where $H$ is a subgroup of $G$ up to conjugation which contains $N$. \\

Given a subgroup $N \lhd H \leq G$ and a $G$-equivariant vector bundle on $G/H$ (i.e. a $H$-representation) $V$ we have a notion of a mark on it. Namely, let $N \lhd K \leq G$ be another subgroup of $G$ containing $N$, and let $k \in K$, we define the mark $(G/H,V)(K,k)$ by 
\[(G/H,V)(K,k) = \sum_{g \in (G/H)^K } \chi_V(g^{-1}kg)\]
where, by a slight abuse of notation, the sum over all $g \in (G/H)^K$ means the sum over representatives in $G$ for each element in $(G/H)^K$. Furthermore, we can define addition and multiplication on the set of pairs $(G/H,V)$. If $(G/H,V)$ and $(G/K,W)$ are two pairs then the addition is given by 
\[(G/H,V) + (G/K,W) = (G/H \cupdot G/K, V \oplus W)\]
and multiplication is given by 
\[(G/H,V) \cdot (G/K, W) = (G/H \times G/K, V \otimes W).\]
These two operations give us a ring structure, however, this ring is not quite what we want. We also wish to quotient out the ideal generated by elements of the form $(G/H \cupdot G/H, V) - (G/H,V \oplus V)$. We call the resulting ring the reduced global representation ring $R(G,N)$. We remark that in the case $N = \{e\}$ we recover the global representation ring discussed by Witherspoon \cite{W}, which is also the same as the generalised Burnside ring of $G$ with respect to the functor that associates its character ring to a subgroup $H$ \cite{GRR}.\\

Now, just as in the case of the Burnside ring, we may store this data in a table. Rows are labelled by pairs consisting of a subgroup $H$ of $G$ containing $N$ along with a $H-$representation. Columns will be labelled by subgroups $K$ of $G$ containing $N$ along with an element of $K$. We adopt the convention that subgroups should be listed in increasing size, so the identity subgroup comes first, and the whole group is last. If two or more subgroups have the same size they can be listed in any order. Elements within subgroups can be listed in any order, so long as the identity comes first. The representations should be listed with increasing dimension, and the trivial representation should appear first. \\

However, we do not need to store all this data in a table, it is straightforward to check that if $k,k' \in K$ are conjugate by an element in $N_G(K)$ then $(G/H,V)(K,k) = (G/H,V)(K,k')$ for all subgroups $H$ and $H$-representations $V$. Thus it is enough to label columns by the subgroups $K$ containing $N$ and all elements of $K$ up to $N_G(K)$ conjugation. In a similar vein if $V$ and $V'$ are two $H$-representations which are equal up to $N_G(H)$-conjugation, that is, there exists a $x \in N_G(H)$ such that $\chi_V(x^{-1}hx) = \chi_{V'}(h)$ for all $h \in H$ then all marks will be equal. Thus, it is enough to label rows by subgroups $H$ of $G$ containing $N$ along with representations of $H$ up to $N_G(H)$-conjugation. We denote the table of the reduced representation ring $R(G,N)$ by $T(G,N)$. Finally, note that once we have fixed a subgroup $K$ and element $k \in K$ the mark \[(-,-)(K,k): \{\text{Subgroups }H: N \lhd H\} \times \{\text{Irreducible representations of }H\text{ up to conjugation by }G\} \longrightarrow \mathbb{C}\] is a ring homomorphism, therefore multiplication in the global representation ring corresponds to multiplication of rows in the global table. 

\begin{example}
Let $G = D_8 = \langle r, s : r^4 = e, s^2 = e, rs = sr^{-1} \rangle$ and let $N = C_2 = \langle r^2 \rangle$. Then the subgroups of $G$ containing $N$ are $N$, $C_4 = \langle r \rangle$, $C_2^2 = \langle r^2, s \rangle$, $\widehat{C_2^2}= \langle r^2, rs \rangle$ and $D_8$. Now, all the subgroups above are normal and so their normalisers are $G$. One calculates that representatives for the classes in each subgroup up to conjugation by $G$ are $\{1,r^2\}$, $\{ 1,r,r^2\}$, $\{1,r^2,s\}$, $\{ 1,r^2,rs\}$ and $\{1,r,r^2,s,rs\}$ respectively. 
Next, we must find representatives of the representations of each subgroup up to conjugation by $G$. In the case of $N = C_2$ the $2$ irreducible representations are not equivalent, label these representations by $1$ and $x$. For $C_4$ the irreducible representations with complex character values are equivalent, let $1$, $y_1$, $y_2$ denote representatives. For $C_2^2$ and $\widehat{C_2^2}$  the irreducible representations which give $-1$ on the class corresponding to $r^2$ are equivalent. In both cases, let $1$,$z_1$, $z_2$ denote representatives. Finally, for $D_8$ the representations will be denoted by $\{1,v_1,v_2,v_3,v_4\}$, where $v_4$ is the $2$-dimensional one. Using our formula for the marks, we produce Table \ref{tab:mytable} with our code \cite{Github}.

\begin{table}[h!]
  \centering\
  \begin{tabular}{|cc|cc|ccc|ccc|ccc|ccccc|}
    \hline
    \multirow{2}{*}{}& & \multicolumn{2}{|c|}{$C_2$} & \multicolumn{3}{|c|}{$C_4$} & \multicolumn{3}{|c|}{$C_2^2$} & \multicolumn{3}{|c|}{$C_2^2$} & \multicolumn{5}{|c|}{$D_8$} \\
    &&1 & $r^2$ & 1 & r & $r^2$ & 1 &$r^2$ & s & 1 &$r^2$ & rs & 1 & r & $r^2$ & s & rs \\
    \hline
    \multirow{2}{*}{\text{$C_2$}} & 1 & 4 & 4 & 0 & 0 & 0 & 0 & 0 & 0 & 0 & 0 & 0 & 0 & 0 & 0 & 0 & 0 \\
    & x & 4 & -4 & 0 & 0 & 0 & 0 & 0 & 0 & 0 & 0 & 0 & 0 & 0 & 0 & 0 & 0 \\
    \hline
    \multirow{3}{*}{$C_4$} & 1 & 2 & 2 & 2 & 2 & 2 & 0 & 0 & 0 & 0 & 0 & 0 & 0 & 0 & 0 & 0 & 0 \\
    & $y_1$ & 2 & -2 & 2 & 0 & -2 & 0 & 0 & 0 & 0 & 0 & 0 & 0 & 0 & 0 & 0 & 0 \\
    & $y_2$ & 2 & 2 & 2 & -2 & 2 & 0 & 0 & 0 & 0 & 0 & 0 & 0 & 0 & 0 & 0 & 0 \\
    \hline
    \multirow{3}{*}{$C_2^2$} & 1 & 2 & 2 & 0 & 0 & 0 & 2 & 2 & 2 & 0 & 0 & 0 & 0 & 0 & 0 & 0 & 0 \\
    & $z_1$ & 2 & 2 & 0 & 0 & 0 & 2 & -2 & 2 & 0 & 0 & 0 & 0 & 0 & 0 & 0 & 0 \\
    & $z_2$ & 2 & -2 & 0 & 0 & 0 & 2 & 0 & -2 & 0 & 0 & 0 & 0 & 0 & 0 & 0 & 0 \\
    \hline
    \multirow{3}{*}{\text{$C_2^2$}} & 1 & 2 & 2 & 0 & 0 & 0 & 0 & 0 & 0 & 2 & 2 & 2 & 0 & 0 & 0 & 0 & 0 \\
    & $z_1$ & 2 & 2 & 0 & 0 & 0 & 0 & 0 & 0 & 2 & 2 & -2 & 0 & 0 & 0 & 0 & 0 \\
    & $z_2$ & 2 & -2 & 0 & 0 & 0 & 0 & 0 & 0 & 2 & -2 & 0 & 0 & 0 & 0 & 0 & 0 \\
    \hline
    \multirow{5}{*}{\text{$D_8$}} & 1 & 1 & 1 & 1 & 1 & 1 & 1 & 1 & 1 & 1 & 1 & 1 & 1 & 1 & 1 & 1 & 1 \\
    & $v_1$ & 1 & 1 & 1 & 1 & 1 & 1 & -1 & 1 & 1 & 1 & -1 & 1 & 1 & -1 & 1 & -1 \\
    & $v_2$ & 1 & 1 & 1 & -1 & 1 & 1 & 1 & 1 & 1 & 1 & -1 & 1 & -1 & 1 & 1 & -1 \\
    & $v_3$ & 1 & 1 & 1 & -1 & 1 & 1 & -1 & 1 & 1 & 1 & 1 & 1 & -1 & -1 & 1 & 1 \\
    & $v_4$ & 2 & -2 & 2 & 0 & -2 & 2 & 0 & -2 & 2 & -2 & 0 & 2 & 0 & 0 & -2 & 0 \\
    \hline
  \end{tabular}
  \caption{Reduced global table of the global representation ring $R(D_8,C_2)$.}
  \label{tab:mytable}
\end{table}
\end{example}

Given a table $T(G,N)$, it is natural to divide the table into blocks labelled by the subgroup $H$ on the row, and the subgroup $K$ on the column, we denote this block by $\mathcal{B}_{H,K}$. We have general descriptions for some of the blocks within the table, as explained in the following lemma.

\begin{proposition} \label{description_of_blocks}
Let $N \lhd G$. The reduced global tables $T(G,N)$ contain the following information:
\begin{enumerate}[label=\roman*)]
    \item The top left entry of each block of $T(G,e)$ recovers the table of marks.
    \item If $S \leq G$ with $N_G(S) = S$ then $\mathcal{B}_{S,S}$ is the character table of $S$. In particular, the bottom right block $\mathcal{B}_{G,G}$ of $T(G,N)$ is the character table of $G$.
    \item $\mathcal{B}_{H,K}$  is a zero block if and only $K$ is not a subgroup of any conjugate of $H$. In particular, the table is block lower triangular.
    \item If $N \lhd G$ then $\mathcal{B}_{N,N}$ is the restriction of the induction of irreducible characters up to deleting duplicate rows and columns.
\end{enumerate}
\end{proposition}

\begin{proof} For $i)$, consider block $\mathcal{B}_{H,K}$. Then the top left entry corresponds to the mark 
    \[(G/H,1)(K,e) = \sum_{g \in (G/H)^K} \chi_{1}(e) = |(G/H)^K|.\] Observe that this is exactly the entry of the table of marks corresponding to subgroups $H,K$.\\
    
    For $ii)$, observe that $(G/S)^S = \{xS \in G/S : x^{-1}Sx \in S\} = N_G(S)/S$. In the case $S = N_G(S)$ then orbits of $S$ under the action of $N_G(S)$ are clearly equal to the conjugacy classes of $S$, hence the columns in the global table corresponding to $S$ are labelled by conjugacy classes. Likewise, the rows corresponding to the subgroup $S$ are labelled by the irreducible representations. Let $V \in \text{Irr}(S)$ be an irreducible representation, and $s \in S$, then $\displaystyle (G/S,V)(K,k) = \sum_{g \in (G/S)^S} \chi_{V}(g^{-1}kg) = \chi_{V}(k)$.\\
    
    For $iii)$, if $K$ is a subgroup of some conjugate of $H$ then the top left entry of the block, which is equal to the corresponding entry in the table of marks, is non-zero and thus the block is not a zero block. On the other hand, suppose $K$ is not a subgroup of any conjugate of $H$, then for any representation $V \in \text{Irr}(H)$ and element $k \in K$ we have 
    $\displaystyle(G/H,V)(K,k) = \sum_{g \in (G/H)^K} \chi_V(g^{-1}xg)$, however, this is an empty sum and so its value is $0$.\\
    
    For $iv)$, let $V \in \text{Irr}(N)$, then the induced representation has character values on elements $n \in N$ given by
    \[\chi_{V\uparrow_{H}^G}(n) = \sum_{g \in G/N} \chi_V(g^{-1}ng) \]
    but this is exactly the value of the mark at entry corresponding to representation $V$ and element $n \in N$. It is left to observe that if $n,n'$ are $G$-conjugates then any induced character takes the same value on $n$ and $n'$. Likewise if $V, W$ are two irreducible representations of $N$ which are conjugate by an element of $G$ then the induced characters of $V$ and $W$ take the same values on all elements of $N$. Thus, inducing all irreducible characters of $N$, restricting them to evaluation only at elements of $N$, and then deleting rows and columns which are duplicates will return exactly the block $\mathcal{B}_{N,N}$, as required.
\end{proof}

Given the global table of a group $G$, it is an interesting question to determine what properties of the group can be recovered from the table (by $\lq\lq$the table" we refer to the inner table, that is, we preserve the values within the table whilst discarding all labelling of rows and columns). For the character table and the table of marks, this has been investigated. For example, given these two tables, we can recover the following properties: the number of conjugacy classes, the isomorphism classes of the centre and abelianisation, the sizes of subgroups and the subgroup lattice (as a directed graph, ordered by inclusion) along with which subgroups are normal. By Proposition~\ref{description_of_blocks} one can also recover these properties from the global table of a group. However, one can also recover additional invariants, as we describe in the following lemma.

\begin{proposition} \label{properties global table}
The following properties of $G$ can be recovered from the global table $T(G,e):$
\begin{enumerate}[label=\roman*)]
    \item All properties of $G$ that can be recovered from the character table and table of marks of $G$.
    \item The order of the elements in any conjugacy class of $G$.
    \item All abelian subgroups of $G$ up to isomorphism class.
\end{enumerate}
\end{proposition}

\begin{proof} For $ii)$, we focus our attention on the last block of rows, that is, the rows corresponding to the subgroup $G \leq G$. Given any element $g \in G$, let $c(g)$ denote its conjugacy class in $G$. One has associated to it its corresponding column in the character table of $G$. Now consider any subgroup $H < G$. The columns appearing in block $\mathcal{B}_{G,H}$ are exactly the columns of the character table of $G$ whose corresponding conjugacy class has a non-zero intersection with $H$. Moreover, if we consider the set $\{H : H < G, H \cap c(g) \neq \phi \}$, which we can locate on the global table as the subgroups $H$ whose block $\mathcal{B}_{G,H}$ contains the column of the character table of $G$ corresponding to $c(g)$, then the subgroup of smallest size in this set must be the cyclic subgroup generated by $g$, and this is precisely the order of $g$.\\

For $iii)$, given a subgroup $H < G$, consider the block $\mathcal{B}_{H,H}$. The entries in the first column are given by $\frac{|N_G(H)|}{|H} \dim(V_i)$ where $V_i$ runs over all irreducible representations of $H$ up to $N_G(H)$ conjugation. Therefore a subgroup $H < G$ is abelian if and only if all its irreducible representations are $1$-dimensional if and only if all entries in the first column of $\mathcal{B}_{H,H}$ are $\frac{|N_G(H)|}{|H}$. Now given an abelian subgroup, by consideration of which columns appear in the blocks $\mathcal{B}_{G,-}$ along with part $b)$ we can determine all its cyclic subgroups, and therefore its isomorphism type.
\end{proof}

It is well-known that there are non-isomorphic groups with the same character table and non-isomorphic groups with the same table of marks. There are also non-isomorphic groups with the same character table and the same table of marks. The first example was constructed in \cite{KR}. Our computations on GAP \cite{Github} show that the smallest example is of order $243$ with GAP IDs $(44,45)$. We have included a table in the appendix with further examples. \\

Recall that the global table contains the character table and table of marks, so the following implications clearly hold.
$$ \text{Same group} \implies \text{Same global table} \implies \text{Same character table and table of marks} $$
 
One may wonder if either of the backward implications holds. This is not the case for either of them. For instance, the previously mentioned groups of order $243$ have the same character table and table of marks, but different global tables. Furthermore, the two pairs of groups of order $256$ with GAP IDs $(1791,1792)$ and $(3678, 3679)$ have the same global table despite not being isomorphic. \\

Observe that the global table ``knows" the character table of each subgroup $H$, albeit ``modulo" the normaliser $N(H)$. Knowing the character tables of cyclic subgroups would determine the power maps on the conjugacy classes. However, the global table does not determine the power maps: the first pair $(1791,1792)$ of groups of order $256$ have the same global tables but are not a Brauer pair, that is, they have different power maps. \\





We finish this section with a brief discussion of the Knutson Index for the reduced global representation ring $R(G,N)$. We have a natural choice for the regular element, namely let $\text{reg}_N$ denote the regular representation of $N$. Then the regular element of $R(G,N)$ corresponds to $(G/N, \text{reg}_N)$. In the corresponding reduced global table $T(G,N)$ this element is of the form $(|G|,0,\dots, 0)$. Observe that when we take $N = G$ and take our generating set to be the rows of the table $T(G,G)$, we recover the original Knutson Index of a group $G$ \cite{Diego}. Furthermore, in the case where we take our generating set to be the rows of the table $T(G,e)$, we have a bound on the Knutson Index given by the Knutson Index on the Burnside ring. Specifically, let $B(G)$ denote the Burnside ring and $T(G)$ its table of marks. Let the regular element of $T(G)$ be given by the row corresponding to $G/\{e\}$, in the table this has row $(|G|,\dots,0)$. Let $T(G)$ have a natural generating set given by the rows of the table $\mathcal{R} = \{r_1, \ldots r_n\}$. Do the same for $T(G,e)$ and denote this generating set $\mathcal{R}'$. We have the following result.

\begin{lemma}
The Knutson Index of the Burnside ring $\mathcal{K}_\mathcal{R} \big(\Omega(G)\big)$ divides the Knutson Index of the global ring $\mathcal{K}_{\mathcal{R}'} \big(R(G,e)\big)$.
\end{lemma}

It would be interesting to study the various Knutson Indices and the global table further. To facilitate future studies, we pose two questions that we find interesting.

\begin{question}
Characterise finite groups with the global Knutson Index $\mathcal{K}_{\mathcal{R}'} \big(R(G,e)\big)$ equal to $1$.   
\end{question}

\begin{question}
Suppose the Knutson Indices of the character ring and Burnside ring of a group $G$ are both equal to $1$. Is the Knutson Index of the global table of $G$ necessarily $1$?
\end{question}



\newpage
\appendix

\section[Further counterexamples to the converse of Lemmas~\ref{lemma p-split}({\em ii}) \& \ref{lemma index}]{\large Further counterexamples to the converse of Lemmas~\ref{lemma p-split}({\em ii}) \& \ref{lemma index}}

The following table contains all non-split extensions in which all resulting sequences of Sylow $p$-subgroups split, where $G$ has order up to and including $383$. This was computed with GAP \cite{Github}. 

\begin{table}[h!]
    \fontsize{9pt}{9pt}\selectfont
    \label{counterexample_table}
    \centering
    \begin{tabular}[t]{c|c}
        Non split extension & GAP ID \\
        \hline
        $Q_8 \cdot C_6$ & $48,33$ \\
        $Q_8 \cdot Dic_3 $ & $96,67$ \\
        $Q_8 \cdot D_{12}$ & $96,193$ \\
        $Q_8 \cdot (C_2 \times C_6)$ & $96,200$ \\
        $(C_2 \times Q_8) \cdot C_6$ & $96,200$ \\
        $Q_8 \cdot A_4$ & $96,201$ \\
        $Q_8 \cdot (C_2 \times C_6)$ & $96,201$ \\
        $(C_4 \circ D_8) \cdot C_6$ & $96,201$ \\
        $Q_8 \cdot C_{18}$ & $144,36$ \\
        $Q_8 \cdot (C_3 \times S_3)$ & $144,127$ \\
        $(C_3 \times Q_8) \cdot C_6$ & $144, 127$ \\
        $Q_8 \cdot (C_3 \times C_6)$ & $144,157$ \\
        $(C_3 \times Q_8) \cdot C_6$ & $144, 157$ \\
        $He_3 \cdot C_6$ & $162, 14$ \\
        $He_3 \cdot S_3$ & $162, 15$ \\
        $(C_2 \cdot C_4^2) \cdot C_6$ & $192, 194$ \\
        $Q_8 \cdot Dic_6$ & $192, 950$ \\
        $Q_8 \cdot (C_4 \times S_3)$ & $192, 953$ \\
        $Q_8 \cdot (C_2 \times Dic_3)$ & $192,981$ \\
        $(C_2 \times Q_8) \cdot Dic_3$ & $192,981$ \\
        $(C_4 \circ D_8) \cdot Dic_3$ & $192,981$ \\
        $(C_2 \times Q_8) \cdot Dic_3$ & $192, 982$ \\
        $(C_4 \circ D_8) \cdot Dic_3$ & $192,982$ \\
        $Q_8 \cdot (C_2 \times Dic_3)$ & $192,985$ \\
        $Q_8 \cdot (C_3 \rtimes D_8)$ & $192, 986$ \\
        $Q_8 \cdot S_4$ & $192,988$ \\
        $Q_8 \cdot (C_3 \rtimes D_8)$ & $192, 988$ \\
        $Q_8 \cdot (C_2 \times C_{12})$ & $192, 997$ \\
        $(C_4 \times Q_8) \cdot C_6$ & $192, 997$ \\
        $Q_8 \cdot (C_2 \times C_{12})$ & $192, 999$ \\
        $(C_4 \times Q_8) \cdot C_6$ & $192, 999$ \\
        $Q_8 \cdot (C_3 \times D_8)$ & $192, 1003$ \\
        $(C_2^2 \times Q_8) \cdot C_6$ & $192, 1003$ \\
        $Q_8 \cdot (C_3 \times D_8)$ & $192, 1005$ \\
        $(C_4 \times Q_8) \cdot C_6$ & $192, 1005$ \\
        $(C_2 \times (C_4 \circ D_8)) \cdot C_6$ & $192, 1005$ \\
        $Q_8 \cdot (C_3 \times Q_8)$ & $192, 1006$ \\
        $Q_{16} \cdot A_4$ & $192,1017$ \\
        $2^{1+4}_+ \cdot C_6$ & $192, 1017$ \\
        $(C_8 \circ D_4) \cdot C_6$ & $192, 1017$ \\
        $2^{1+4}_{-} \cdot C_6$ & $192, 1018$ \\
        $SD_{16} \cdot A_4$ & $192, 1018$ \\
        $Q_8 \cdot (C_2^2 \times S_3)$ & $192,1481$ \\
        $(C_2 \times Q_8) \cdot D_{12}$ & $192,1481$ \\
        $(C_2 \times Q_8) \cdot D_{12}$ & $192,1482$ \\
        $(C_4 \circ D_8) \cdot D_{12}$ & $192,1482$ \\
        $(C_4 \circ D_8) \cdot D_{12}$ & $192,1484$ \\
        $Q_8 \cdot (C_2^2 \times C_6)$ & $192,1500$ \\
        $(C_2 \times Q_8) \cdot (C_2 \times C_6)$ & $192,1500$ \\
        $(C_2^2 \times Q_8) \cdot C_6$ & $192,1500$ \\
        $Q_8 \cdot (C_2 \times A_4)$ & $192,1502$ \\
        $Q_8 \cdot (C_2^2 \times C_6)$ & $192,1502$ \\
        $(C_2 \times Q_8) \cdot A_4$ & $192,1502$ \\
        $(C_2 \times Q_8) \cdot (C_2 \times C_6)$ & $192,1502$ \\
        $(C_4 \circ D_8) \cdot (C_2 \times C_6)$ & $192,1502$ \\
        $(C_2 \times (C_4 \circ D_8)) \cdot C_6$ & $192,1502$ \\
        $(C_2 \times Q_8) \cdot (C_2 \times C_6)$ & $192,1504$ \\
        $(C_4 \circ D_8) \cdot A_4$ & $192,1504$ \\
        $(C_4 \circ D_8) \cdot (C_2 \times C_6)$ & $192,1504$ \\
        $Q_8 \cdot (C_2 \times A_4)$ & $192,1507$ \\
        $(C_2^2 \times Q_8) \cdot C_6$ & $192,1507$ \\
    \end{tabular}
    \hspace{8ex}
    \begin{tabular}[t]{c|c}
        Non split extension & GAP ID \\
        \hline 
        $2_+^{1+4} \cdot C_6$ & $192,1509$ \\
        $Q_8 \cdot (C_3 \times D_{10})$ & $240, 108$ \\
        $(C_5 \times Q_8) \cdot C_6$ & $240,108$ \\
        $Q_8 \cdot C_{30}$ & $240, 154$ \\
        $(C_5 \times Q_8) \cdot C_6$ & $240, 154$ \\
        $Q_8 \cdot Dic_9$ & $288, 70$ \\
        $Q_8 \cdot D_{36}$ & $288,340$ \\
        $Q_8 \cdot (C_2 \times C_{18})$ & $288,347$ \\
        $(C_2 \times Q_8) \cdot C_{18}$ & $288,347$ \\
        $Q_8 \cdot (C_3 \cdot A_4)$ & $288,348$ \\
        $Q_8 \cdot (C_2 \times C_{18})$ & $288,348$ \\
        $(C_4 \circ D_8) \cdot C_{18}$ & $288,348$ \\
        $Q_8 \cdot (C_3 \times Dic_3)$ & $288,400$ \\
        $(C_3 \times Q_8) \cdot Dic_3$ & $288,400$ \\
        $Q_8 \cdot (C_2 \rtimes Dic_3)$ & $288,404$ \\
        $(C_3 \times Q_8) \cdot Dic_3$ & $288,404$ \\
        $(C_3 \times Q_8) \cdot D_{12}$ & $288,847$ \\
        $Q_8 \cdot (C_3 \times A_4)$ & $288,860$ \\
        $Q_8 \cdot (S_3 \times C_6)$ & $288,904$ \\
        $(C_3 \times Q_8) \cdot D_{12}$ & $288,904$ \\
        $Q_8 \cdot (C_2 \times C_3 \rtimes S_3)$ & $288,915$ \\
        $(C_3 \times Q_8) \cdot D_{12}$ & $288,915$ \\
        $Q_8 \cdot (S_3 \times C_6)$ & $288,921$ \\
        $(C_2 \times Q_8) \cdot (C_3 \times S_3)$ & $288,921$ \\
        $(C_3 \times Q_8) \cdot (C_2 \times C_6)$ & $288,921$ \\
        $(C_6 \times Q_8) \cdot C_6$ & $288,921$ \\
        $Q_8 \cdot (S_3 \times C_6)$ & $288,924$ \\
        $(C_4 \circ D_8) \cdot (C_3 \times S_3)$ & $288,924$ \\
        $(C_3 \times Q_8) \cdot (C_2 \times C_6)$ & $288,924$ \\
        $Dic_6 \cdot A_4$ & $288,924$ \\
        $(Q_8 \rtimes S_3) \cdot C_6$ & $288,924$ \\
        $(C_3 \times (C_4 \circ D_8)) \cdot C_6$ & $288,924$ \\
        $Q_8 \cdot (S_3 \times C_6)$ & $288,925$ \\
        $(C_3 \times Q_8) \cdot (C_2 \times C_6)$ & $288,925$ \\
        $(S_3 \times Q_8) \cdot C_6$ & $288,925$ \\
        $Q_8 \cdot C_6^2$ & $288,983$ \\
        $(C_2 \times Q_8) \cdot (C_3 \times C_6)$ & $288,983$ \\
        $(C_3 \times Q_8) \cdot (C_2 \times C_6)$ & $288,983$ \\
        $(C_6 \times Q_8) \cdot C_6$ & $288,983$ \\
        $(C_3 \times (C_4 \circ D_8)) \cdot C_6$ & $288,984$ \\
        $(C_4 \circ D_8) \cdot (C_3 \times C_6)$ & $288,984$ \\
        $(C_3 \times Q_8) \cdot A_4$ & $288,984$ \\
        $(C_3 \times Q_8) \cdot (C_2 \times C_6)$ & $288,984$ \\
        $Q_8 \cdot (C_3 \times A_4)$ & $288,984$ \\
        $Q_8 \cdot C_6^2$ & $288,984$ \\        
        $2^{1+4}_- \cdot C_{10}$ & $320,1586$ \\
        $He_3 \cdot C_{12}$ & $324, 17$ \\
        $He_3 \cdot Dic_3$ & $324, 18$ \\
        $He_3 \cdot D_{12}$ & $324, 41$ \\
        $He_3 \cdot (C_2 \times C_6)$ & $324, 72$ \\
        $(C_2 \times He_3) \cdot C_6$ & $324, 72$ \\
        $He_3 \cdot D_{12}$ & $324, 73$ \\
        $(C_2 \times He_3) \cdot S_3$ & $324, 73$ \\
        $Q_8 \cdot (C_3 \times D_{14})$ & $336,131$ \\
        $(C_7 \times Q_8) \cdot C_6$ & $336, 131$ \\
        $Q_8 \cdot Dic_{7}$ & $336,134$ \\
        $(C_7 \times Q_8) \cdot C_6$ & $336, 134$ \\
        $Q_8 \cdot C_{42}$ & $336,170$ \\
        $(C_7 \times Q_8) \cdot C_6$ & $336, 170$ \\
        $Q_8 \cdot (C_3 \times C_7 \rtimes C_3)$ & $336,173$ \\
        $(C_7 \times Q_8) \cdot C_6$ & $336, 173$ \\
    \end{tabular}
    \label{table}
\end{table}

\section{Groups with same character table and same table of marks}\label{appendix-samecharburn}
Up to order $765$ there are examples of such pairs of groups of orders $243$, $256$, $384$, $486$, $512$, $640$, $672$ and $729$. Note that the lists are potentially not exhaustive for the groups of order $256$ and $640$ because we have stopped the calculation. We have not attempted it for groups of order $512$. This was computed with GAP \cite{Github}. 

\begin{table}[h!]
    \label{same_char_same_burnside}
    \centering
    \begin{tabular}[t]{c|l}
    Order of groups & \hspace{35mm} GAP IDs \\
    \hline
    \\
    $243 = 3^5$ & $(44,45)$ \\
    \\
    $256 = 2^8$ & $(227,228)$, \ $(1127,1128)$, \ $(1129,1130)$, \ $(1720,1721)$, \ $(1728,1729)$, \\ & $(1732,1733)$, \ $(1741,1742)$, \ $(1770,1771)$, \ $(1791,1792)$, \ $(3599,3600)$, \\ & $(3601,3602)$, \ $(3678,3679)$, \ $(4156,4159)$, \ $(4530,4534)$, \ $(4532,4536)$ \\ & $(10818,10820)$, \ $(10819,10821)$, \ $(10971,10973)$, \ $(10972,10974)$, \\ & $(10994,10996)$, \ $(10998,11000)$, \ $(11064,11066)$, \ $(11065,11067)$, \\ & $(11068,11070)$, \ $(11069,11071)$, \ $(15875,15877)$, \ $(15937,15939)$, \\ & $(53284,53285)$, \ $(53306,53315)$, \ $(53307,53314)$, \ $(53322,53327)$, \\ & $(53324,53332)$, $(53326,53334)$  \\  
    \\
    $384 = 2^7 \cdot 3$ & $(12031, 12095)$, \ 
    $(12035,12096)$, \ 
    $(12041,12101)$, \ 
    $(12043,12104)$, \\
    & $(12055,12113)$, \ 
    $(12058,12115)$, \ 
    $(12064,12120)$, \ 
    $(12066,12121)$ \\
    \\
    $486 = 2 \cdot 3^5$ & $(155,157)$, \ $(203,204)$ \\
    \\
    $512 = 2^9$ & \textcolor{red}{not computed} \\
    \\
    $640 = 2^7 \cdot 5$ & $(13268,13332)$, \
     $(13272,13333)$, \    
     $(13278,13338)$, \
     $(13280,13341)$, \\
    & $(13292,13350)$, \
     $(13295,13352)$, \
     $(13301,13357)$, \
     $(13303,13358)$ \\
    & \textcolor{red}{potentially more} \\ 
    \\
    $672 = 2^5 \cdot 3 \cdot 7$ & $(756,757)$ \\
    \\
    $729 = 3^6$ & $(14,15)$, $(31,33)$, $(44,45)$, $(46,47)$, $(56,57)$, $(72,73)$, $(82,88)$, \\
    & $(111,112)$, $(128,129)$, $(222,233)$, $(224,237)$, $(226,228)$, $(232,236)$, \\
    & $(262,264)$, $(288,289)$, $(330,334)$, $(332,336)$, $(434,435)$, $(464,467)$ 
    \end{tabular}
    \label{table_samecharburn}
\end{table}

\newpage
\printbibliography

@misc{Github,
  title={Global representation ring complementary code},
  author={Johnston, Dylan and Mart{\'\i}n Duro, Diego},
  url = {https://github.com/global-rep-ring/complementary_code},
  note={Accessed: 24/04/2024},
  year={2024}
}

@book{K,
  title = {$\lambda$-Rings and the representation theory of the symmetric group},
  author = {Knutson, Donald I.},
  year = {1973},
  publisher = {Springer}
}

@article{S,
   title = {Knutson's Conjecture on the irreducible characters of finite groups},
   author = {Savitskii, V. G.},
   journal = {Bulletin of Moscow University},
   year = {1992},
   volume = {Series 1: Mathematics}
}

@article {KR,
    AUTHOR = {Kimmerle, Wolfgang and Roggenkamp, Klaus W.},
     TITLE = {Non-isomorphic groups with isomorphic spectral tables and
              {B}urnside matrices},
  JOURNAL = {Chinese Annals of Mathematics. Series B},
    VOLUME = {15},
      YEAR = {1994},
    NUMBER = {3},
     PAGES = {273--282},
}

@article{Se,
  title={Equivariant $ K $-theory},
  author={Segal, Graeme},
  journal={Publications Math{\'e}matiques de l'IH{\'E}S},
  volume={34},
  pages={129--151},
  year={1968}
}

@article{Diego,
  title={The Knutson Index of the representation ring},
  author={Mart{\'\i}n Duro, Diego},
  journal={Journal of Algebra},
  volume={659},
  pages={516--541},
  year={2024}
}

@article{Diego2,
  title={Arithmetic properties of character degrees and the generalised Knutson Index},
  author={Mart{\'\i}n Duro, Diego},
  journal={Preprint, arXiv:2306.12408},
  year={2023}
}

@article {GRR,
    AUTHOR = {Gunnells, Paul E. and Rose, Andrew and Rumynin, Dmitriy},
    TITLE = {Generalised {B}urnside rings, {$G$}-categories and module
              categories},
   JOURNAL = {Journal of Algebra},
    VOLUME = {358},
    YEAR = {2012},
    PAGES = {33--50}
}

@article {N,
    AUTHOR = {Nenciu, Adriana},
     TITLE = {Brauer {$t$}-tuples},
   JOURNAL = {Journal of Algebra},
    VOLUME = {322},
      YEAR = {2009},
    NUMBER = {2},
     PAGES = {410--428},
}

@article{W,
  title={The ring of equivariant vector bundles on finite sets},
  author={Witherspoon, Sarah},
  journal={Journal of Algebra},
  volume={175},
  number={1},
  pages={274--286},
  year={1995},
  publisher={Elsevier}
}

@article{Sa,
  title={On the converse of Gasch{\"u}tz’complement theorem},
  author={Sambale, Benjamin},
  journal={Journal of Group Theory},
  volume={26},
  number={5},
  pages={931--949},
  year={2023},
  publisher={De Gruyter}
}

@article{Sh,
  title={The existence of $\pi$-complements to normal subgroups of finite groups},
  author={Shemetkov, Leonid A.},
  journal={Proceedings of the USSR Academy of Sciences},
  volume={195},
  number={1},
  pages={50--52},
  year={1970},
  publisher={Russian Academy of Sciences}
}

\end{document}